\documentclass[12pt]{amsart}
\usepackage{amssymb}
\usepackage{amsmath, amscd}
\usepackage{amsthm}
\usepackage{comment}
\usepackage[table,xcdraw]{xcolor}
\usepackage{ulem}
\usepackage{tikz}
\usepackage{float}
\usepackage{graphicx}
\usepackage{circuitikz}
\usetikzlibrary{positioning}
\usepackage[colorlinks=true, linkcolor=blue,urlcolor=blue]{hyperref}

\newtheorem{theorem}{Theorem}[section]

\newtheorem{proposition}[theorem]{Proposition}
\newtheorem{lemma}[theorem]{Lemma}

\newtheorem{corollary}[theorem]{Corollary}
\theoremstyle{definition}

\newtheorem{example}[theorem]{Example}
\newtheorem{definition}[theorem]{Definition}

\begin{document}
	
\author[P. Danchev]{Peter Danchev}
\address{Institute of Mathematics and Informatics, Bulgarian Academy of Sciences, 1113 Sofia, Bulgaria}
\email{danchev@math.bas.bg; pvdanchev@yahoo.com}
\author[M. Doostalizadeh]{Mina Doostalizadeh}
\address{Department of Mathematics, Tarbiat Modares University, 14115-111 Tehran Jalal AleAhmad Nasr, Iran}
\email{d\_mina@modares.ac.ir;  m.doostalizadeh@gmail.com}
\author[O. Hasanzadeh]{Omid Hasanzadeh}
\address{Department of Mathematics, Tarbiat Modares University, 14115-111 Tehran Jalal AleAhmad Nasr, Iran}
\email{o.hasanzade@modares.ac.ir; hasanzadeomiid@gmail.com}
\author[A. Javan]{Arash Javan}
\address{Department of Mathematics, Tarbiat Modares University, 14115-111 Tehran Jalal AleAhmad Nasr, Iran}
\email{a.darajavan@modares.ac.ir; a.darajavan@gmail.com}
\author[A. Moussavi]{Ahmad Moussavi}
\address{Department of Mathematics, Tarbiat Modares University, 14115-111 Tehran Jalal AleAhmad Nasr, Iran}
\email{moussavi.a@modares.ac.ir; moussavi.a@gmail.com}

\title[Rings whose non-units are strongly weakly nil-clean]{Rings Whose Non-Invertible Elements are Strongly Weakly Nil-Clean}

\maketitle




\begin{abstract}
The target of the present work is to give a new insight in the theory of {\it strongly weakly nil-clean} rings, recently defined by Kosan and Zhou in the Front. Math. China (2016) and further explored in detail by
Chen-Sheibani in the J. Algebra Appl. (2017). Indeed, we consider those rings whose non-units are strongly weakly nil-clean and succeed to establish that this class of rings is strongly $\pi$-regular and, even something more, that it possesses a complete characterization in terms of the Jacobson radical and sections of the $2\times 2$ full matrix ring. Additionally, some extensions like Morita context rings and groups rings are also studied in this directory.
\end{abstract}

\section{Introduction and Basic Concepts}

The motivation of this paper is the class of {\it strongly weakly nil-clean} rings as well as results based on a recent work by Danchev et al. on GSNC rings (cf. \cite{5}, \cite{29} and \cite{30}). In fact, we will combine in what follows the used by us approach developed in \cite{5} in order to obtain a complete description of a widely more large class of rings.

Everywhere in the current paper, let $R$ be an associative but {\it not} necessarily commutative ring having identity element, usually stated as $1$. Standardly, for such a ring $R$, the letters $U(R)$, $Nil(R)$ and $Id(R)$ are designed for the set of invertible elements (also termed as the unit group of $R$), the set of nilpotent elements and the set of idempotent elements in $R$, respectively. Likewise, $J(R)$ denotes the Jacobson radical of $R$, and $Z(R)$ denotes the center of $R$. The ring of $n\times n$ matrices over $R$ and the ring of $n\times n$ upper triangular matrices over $R$ are, respectively, denoted by $M_{n}(R)$ and $T_{n}(R)$. We also denote by $R[x]$ and $R[[x]]$ the polynomial ring and power series ring over $R$, respectively. Standardly, a ring is said to be {\it abelian} if each of its idempotents is central.

Historically, Nicholson introduced in \cite{16} and \cite{17} the notion of {\it clean} and {\it strongly clean} rings. A ring is called (strongly) clean if every its element can be written as sum of an unit element and an idempotent element (that commute). In the past two decades, there has been intense interest in this class of rings as well as certain generalizations of their induced variants were produced.

Later on, as a different variant of clean rings, Ko{\c{s}}an et al. introduced \cite{19} the notion of {\it weakly clean} rings in general aspect (note that the commutative variant of weak cleanness was firstly defined by Ahn-Anderson in \cite{AA}). A ring $R$ is called (strongly) weakly clean if every its element can be written as sum or difference an unit element and an idempotent element (that commute). The authors in \cite{19} proved many fundamental properties of weakly clean rings.

Likewise, Diesl introduced in \cite{2} the notion of {\it nil-clean} rings. A ring is said to be (strongly) nil-clean if, for each $r\in R$, there are $q\in Nil(R)$ and $e\in Id(R)$ such that $r=q+e$ ($qe=eq$). He obtained many fundamental properties as well as developed a general theory of (strongly) nil-clean rings.

As a different variant of nil-clean rings, Danchev and McGovern introduced in \cite{21} the notion of weak nil-cleanness for commutative rings. A commutative ring $R$ is called {\it weakly nil-clean} if, for each $r\in R$, there are $q\in Nil(R)$ and $e\in Id(R)$ such that $r=q+e$ or $r=q-e$. The authors proved many fundamental properties of commutative weakly nil-clean rings. They also provided a complete characterization theorem of commutative weakly nil-clean rings and proved a characterization theorem of commutative weakly nil-clean group rings. Furthermore, Breaz et al. defined in \cite{20} the notion of weak nil-cleanness for arbitrary rings as that in \cite{21}, and they also obtained a decomposition theorem for abelian rings to be weakly nil-clean. They also proved that, for a division ring $D$ and a natural number $n$, $M_n(D)$ is weakly nil-clean if, and only if, either $D \cong \mathbb{Z}_2$, or $D \cong \mathbb{Z}_3$ and $n=1$.

As a common generalization of the abelian weakly nil-clean rings, Kosan and Zhou introduced in \cite{28} the concept of {\it strongly weakly nil-clean} rings. A ring $R$ is called strongly weakly nil-clean if, for each $r\in R$, $r=q+e$ or $r=q-e$, where $q\in Nil(R)$, $e\in Id(R)$ and $qe=eq$. In that article both the structural theorems about commutative weakly nil-clean rings and abelian weakly nil-clean rings were somewhat extended.

In the other vein, a ring $R$ is called $UU$ if all of its units are unipotent, that is, $U(R) \subseteq 1+{\rm Nil}(R)$ (and so $1+{\rm Nil}(R)=U(R)$). This notion was introduced by C{\u{a}}lug{\u{a}}reanu \cite{23}, and next examined in more details by Danchev and Lam in \cite{22}. They showed that $R$ is strongly nil-clean if, and only if, $R$ is an exchange $UU$ ring. Besides, in a more general context, a ring $R$ is called {\it weakly UU}, and abbreviated as $WUU$, if $U(R)={\rm Nil}(R) \pm 1$. This is equivalent to the condition that every unit can be presented as either $n+1$ or $n-1$, where $n \in {\rm Nil}(R)$ (see the original source \cite{24} too).

Due to Danchev et al. (see \cite{5}), a ring $R$ is called {\it generalized strongly nil-clean}, or just briefly abbreviated by GSNC, if every non-invertible element in $R$ is strongly nil-clean. The authors in \cite{5} proved many fundamental properties of GSNC rings. They also provided a characterization theorem of semi-local GSNC rings and proved a characterization theorem of GSNC group rings. One of the results in \cite{5} states that a ring $R$ is strongly nil-clean if, and only if, $R$ is strongly clean and $UU$.

Now, considering the strongly weakly nil-clean rings from \cite{1} and \cite{28}, in parallel to the aforementioned GSNC rings, we introduce in a way of similarity the GSWNC rings, i.e., those rings whose non-invertible elements are strongly weakly nil-clean. These rings are simultaneously a non-trivial generalization of both strongly weakly nil-clean rings and GSNC rings, and also is a proper subclass of the class of strongly weakly clean rings as mentioned above.

So, the leitmotif of the present article is various non-elementary expansions of these rings to be studied in the sequel in-depth. Our central achievemnets are Theorems~\ref{2.25}, \ref{2.38} and \cite{theorem 2.36}, respectively. Some applications of these results to GSWNC group rings are provided in Theorems~\ref{3.12} and \ref{3.3} as well.

\section{GSWNC rings}

We begin here with our key instrument.

\begin{definition}\label{2.1}
A ring $R$ is called {\it generalized strongly weakly nil-clean} (or just {\it GSWNC} for short) if every non-invertible element in $R$ is strongly weakly nil-clean.	
\end{definition}

We now need a series of preliminaries.

\begin{lemma}\label{2.2}
Let $R$ be a ring and $a\in R$ be a strongly weakly nil-clean element. Then, $-a$ is strongly weakly-clean.
\end{lemma}

\begin{proof}
Assume $a = q \pm e$ is a strongly weakly nil-clean representation. If $a=q+e$, then we have, $-a=(1-e)-(q+1)$ where $1-e$ is an idempotent and $q+1$ is an unit in $R$. If $a=q-e$, then we have $-a=-(1-e)+(1-q)$ where again $1-e$ is an idempotent and $1-q$ is an unit in $R$. Thus, $-a$ has a strongly weakly-clean decomposition and hence it is strongly weakly-clean element.
\end{proof}

An automatic consequence is the following obvious claim.

\begin{corollary}\label{2.3}
Every GSWNC ring is strongly weakly-clean.
\end{corollary}

In other words, GSWNC rings lie between strongly weakly nil-clean rings and strongly weakly clean rings. However, it is worthwhile noticing that $\mathbb{Z}_3\times \mathbb{Z}_3$ and $M_2(\mathbb{Z}_3)$ and $M_2(\mathbb{Z}_2)$ are all GSWNC rings that are manifestly not strongly weakly nil-clean, while $\mathbb{Z}_{2}[[x]]$ and $\mathbb{Z}_{3}[[x]]$ are strongly weakly-clean rings that are not GSWNC. Also, $M_2(\mathbb{Z}_3)$ and $\mathbb{Z}_3\times \mathbb{Z}_3$ and $\mathbb{Z}_3\times \mathbb{Z}_2$ are all GSWNC rings, but are not GSNC.

\begin{lemma}\label{2.4}
Let $R_{i}$ be a ring for all $i\in I$. If $\prod^{n}_{i=1}R_{i}$ is GSWNC, then each $R_{i}$ is GSWNC.
\end{lemma}

\begin{proof}
Let $a_{i}\in R_{i}$, where $a_{i}\notin U(R_{i})$, whence $ (1,1,\ldots,a_{i},1,\ldots,1)\notin  U(\prod^{n}_{i=1}R_{i}) $. So, $(1,1,\ldots,a_{i},1,\ldots,1)$ is strongly weakly nil-clean, and hence $a_{i}$ is strongly weakly nil-clean. If, however, $a_{i}$ is not strongly weakly nil-clean, we clearly conclude that $(1,1,\ldots,a_{i},1,\ldots,1)$
is not strongly weakly nil-clean and this is a contradiction.
\end{proof}

\begin{proposition}\label{2.5}
Let $R$ and $S$ be rings. If $R\times S$ is GSWNC, then $R$ and $S$ are strongly weakly nil-clean.
\end{proposition}

\begin{proof}
Let $a\in R$ is an arbitrary element, so $(a,0) \in R \times S$ is not unit. Then by hypothesis we have $(a,0)=(q,0)\pm(e,0)$ where $(q,0)$ is a nilpotent and $(e,0)$ is an idempotent in $R \times S$ and $(q,0)(e,0)=(e,0)(q,0)$. Hence $a=q\pm e$ where $q$ is a nilpotent and $e$ is an idempotent in $R$ and $qe=eq$. So $a$ is a strongly weakly nil-clean element. Then $R$ is strongly weakly nil-clean ring. Similarly $S$ is strongly weakly nil-clean ring.
\end{proof}

\begin{proposition}\label{2.6}
The direct product $\prod_{i=1}^{n} R_i$ is GSWNC for $n \ge 3$ if, and only if, each $R_i$ is strongly weakly nil-clean and at most one of them is {\it not} strongly nil-clean.
\end{proposition}

\begin{proof}
($\Rightarrow$). Assume \(\prod_{i=1}^{n} R_i\) is a GSWNC ring. Therefore, with Proposition \ref{2.5} at hand, \(\prod_{i=1}^{n-1} R_i\) and \(R_n\) are strongly weakly nil-clean rings. Thus, according to \cite[Theorem 2.2]{1}, without loss of generality, we can assume that for each \(1 \le i \le n-2\), \(R_i\) is a strongly nil-clean ring. Again, since \(\prod_{i=1}^{n} R_i = (R_1 \times \cdots \times R_{n-2}) \times (R_{n-1} \times R_n)\), Proposition \ref{2.5}, implies that \(R_{n-1} \times R_n\) is strongly weakly nil-clean. Therefore, by \cite[Theorem 2.2]{1}, we can assume that \(R_{n-1}\) is strongly nil-clean and \(R_n\) is strongly weakly nil-clean.\\
($\Leftarrow$). It is clear owing to Theorem 2.2 proved in \cite{1}.
\end{proof}

\begin{example}
The ring $\mathbb{Z}_3 \times \mathbb{Z}_3$ is GSWNC, but is {\it not} strongly weakly nil-clean.	
\end{example}	

\begin{corollary}\label{2.7}
Let $L=\prod_{i \in I} R_i$ be the direct product of rings $R_i \cong R$ and $|I| \geq 3$. Then, $L$ is a GSWNC ring if, and only if, $L$ is a GSNC ring if, and only if, $L$ is strongly nil-clean if, and only if, $R$ is strongly nil-clean.
\end{corollary}

\begin{corollary}\label{2.8}
For any $n \geq 3$, the ring $R^n$ is GSWNC if, and only if, $R^n$ is GSNC if, and only if, $R$ is strongly nil-clean.
\end{corollary}

\begin{proposition}\label{2.9}
Let $R$ be a ring. Then, the following are equivalent:
\begin{enumerate}
\item
$R$ is strongly nil-clean.
\item
${\rm T}_{n}(R)$ is strongly weakly nil-clean for all $n \in \mathbb{N}$.
\item
${\rm T}_n(R)$ is strongly weakly nil-clean for some $n \geq 3$.
\item
${\rm T}_n(R)$ is GSWNC for some $n \geq 3$.
\end{enumerate}
\end{proposition}

\begin{proof}
(i) $\Rightarrow$ (ii). This follows employing \cite[Theorem 4.1]{2}.\\
(ii) $\Rightarrow$ (iii) $\Rightarrow$ (iv). These are trivial.\\
(iv) $\Rightarrow$ (i). Let $a\in R$. Choose
$$A=\begin{pmatrix}
	a & & & \ast \\
	& -a_{1} & & \\
	& & 0 & \\
	& & & \ddots & \\
	0 & & & & 0
\end{pmatrix}\in {\rm T}_n(R).$$ It is clear that $A$ is non-invertible in ${\rm T}_n(R)$. By hypothesis, we can find an idempotent
$E=\begin{pmatrix}
	e_{1} & & & \ast \\
	& e_{2} & & \\
	& & \ddots & \\
	0 & & & e_{n}
\end{pmatrix}$ and a nilpotent \\
$Q=\begin{pmatrix}
	q_{1} & & & \ast \\
	& q_{2} & & \\
	& & \ddots & \\
	0 & & & q_{n}
\end{pmatrix}$ such that
$$A=\pm \begin{pmatrix}
	e_{1} & & & \ast \\
	& e_{2} & & \\
	& & \ddots & \\
	0 & & & e_{n}
\end{pmatrix}+\begin{pmatrix}
	q_{1} & & & \ast \\
	& q_{2} & & \\
	& & \ddots & \\
	0 & & & q_{n}
\end{pmatrix}$$
and $EQ=QE$. It now follows that $a=e_1+q_1$ or $a=e_2-q_2$, $e_1q_1=q_1e_1$ and $e_2q_2=q_2e_2$. Clearly, $e_1$, $e_2$ are idempotents and $q_1$, $q_2$ are nilpotents in $R$. Thus, $a$ has strongly nil-clean decomposition, proving (i).\\
\end{proof}

\begin{example}
The ring $T_{2}(\mathbb{Z}_3)$ is GSWNC, but $\mathbb{Z}_3$ is {\it not} strongly nil-clean.	
\end{example}

\begin{lemma}\label{2.10}
If $R$ is GSWNC, then $J(R)$ is nil.
\end{lemma}

\begin{proof}
Choose $j \in J(R)$. Since $j \notin U(R)$, we have $e = e^2 \in R$ and $q \in Nil(R)$ such that $j = q \pm e$ where $qe=eq$. Therefore, $1-e = (q+1) - j \in U(R) + J(R) \subseteq U(R)$ or $1-e = (1-q) + j \in U(R) + J(R) \subseteq U(R)$ , so $e = 0$. Hence, $j = q \in Nil(R)$, as required.
\end{proof}

\begin{proposition}\label{2.11}
Let $R$ be a ring, and let $a\in R$. Then, the following are equivalent:
\\
(i) $R$ is a GSWNC ring.
\\
(ii) For any $a\in R$, where $a\notin U(R)$, $a\pm a^{2}\in {\rm Nil}(R)$.
\end{proposition}

\begin{proof}
The proof is quite analogous to \cite[Theorem 2.1]{1}.
\end{proof}

A ring $R$ is said to be {\it strongly $\pi$-regular}, provided that, for any $a \in R$, there exists $n \in \mathbb{N}$ such that $a^n \in a^{n+1}R$.

\medskip

The following observation is crucial.

\begin{corollary}\label{2.12}
Every GSWNC ring is strongly $\pi$-regular.
\end{corollary}

\begin{proof}
Let $R$ be a GSWNC ring. Let $a \in R$. If $a \in U(R)$, then $a$ is a strongly $\pi$-regular element. If, however, $a \notin U(R)$, then, by Proposition \ref{2.11}, there exists $m \in \mathbb{N}$ such that $(a\pm a^2)^m = 0$, so we have $a^m = a^{m+1}r$ for some $r \in R$. Thus, $R$ is a strongly $\pi$-regular ring, as claimed.
\end{proof}	

\begin{proposition}\label{2.13}
Let $R$ be a ring. Then:	
\\
(i) For any nil-ideal $I \subseteq R$, $R$ is GSWNC if, and only if, $R/I$ is GSWNC.
\\
(ii) A ring $R$ is GSWNC if, and only if, $J(R)$ is nil and $R/J(R)$ is GSWNC.
\end{proposition}

\begin{proof}
(i) Let $R$ is a GSWNC ring and $\overline{R} := R/I$, where $\bar{a} \notin U(\overline{R})$. Then, $a \notin U(R)$, which insures, in view of Proposition \ref{2.11}, that $a \pm a^2 \in {\rm Nil}(R)$, so $\bar{a} \pm \bar{a}^2 \in {\rm Nil}(\overline{R})$.

Conversely, suppose $\overline{R}$ is a GSWNC ring. If $a \notin U(R)$, then $\bar{a} \notin U(\overline{R})$, and Proposition \ref{2.11} yields that $\bar{a} \pm \bar{a}^2 \in {\rm Nil}(\overline{R})$. Therefore, there exists $k \in \mathbb{N}$ such that $(a \pm a^2)^k \subseteq I \subseteq {\rm Nil}(R)$.

(ii) Using Lemma \ref{2.10} and part (i) of the proof, the proof is clear.
\end{proof}

\begin{lemma}\label{2.14}
Let $R$ be a ring and $0\neq e\in {\rm Id}(R)$. If $R$ is GSWNC, then so is $eRe$.
\end{lemma}

\begin{proof}
Let $a \in eRe$ where it is not unit in $eRe$. We have $a = ea = ae = eae$. If $a \in U(R)$, then there exists $b \in R$ such that $ab = ba = 1$, which implies $a(ebe) = (ebe)a = e$, a contradiction. Therefore, $a \notin U(R)$. Hence, by Proposition \ref{2.11}, we have $a \pm a^2 \in {\rm Nil}(R) \cap eRe \subseteq {\rm Nil}(eRe)$, as required.
\end{proof}

Let $R$ be a ring and $M$ a bi-module over $R$. The trivial extension of $R$ and $M$ is defined as
\[ T(R, M) = \{(r, m) : r \in R \text{ and } m \in M\}, \]
with addition defined componentwise and multiplication defined by
\[ (r, m)(s, n) = (rs, rn + ms). \]

\begin{corollary}\label{2.17}
Let $R$ be a ring and $M$ a bi-module over $R$. Then, the following are equivalent:
\begin{enumerate}
\item
$R$ is a GSWNC ring.
\item
$T(R, M)$ is a GSWNC ring.
\end{enumerate}
\end{corollary}

\begin{proof}
Set $A={\rm T}(R, M)$ and consider $I:={\rm T}(0, M)$. It is not too hard to verify that $I$ is a nil-ideal of $A$ such that $\dfrac{A}{I} \cong R$. So, the result follows directly from Proposition \ref{2.13}.
\end{proof}

Let $\alpha$ be an endomorphism of $R$ and $n$ a positive integer. It was defined by Nasr-Isfahani in \cite{3} the {\it skew triangular matrix ring} like this:
$${\rm T}_{n}(R,\alpha )=\left\{ \left. \begin{pmatrix}
	a_{0} & a_{1} & a_{2} & \cdots & a_{n-1} \\
	0 & a_{0} & a_{1} & \cdots & a_{n-2} \\
	0 & 0 & a_{0} & \cdots & a_{n-3} \\
	\ddots & \ddots & \ddots & \vdots & \ddots \\
	0 & 0 & 0 & \cdots & a_{0}
\end{pmatrix} \right| a_{i}\in R \right\}$$
with addition point-wise and multiplication given by:
\begin{align*}
&\begin{pmatrix}
		a_{0} & a_{1} & a_{2} & \cdots & a_{n-1} \\
		0 & a_{0} & a_{1} & \cdots & a_{n-2} \\
		0 & 0 & a_{0} & \cdots & a_{n-3} \\
		\ddots & \ddots & \ddots & \vdots & \ddots \\
		0 & 0 & 0 & \cdots & a_{0}
	\end{pmatrix}\begin{pmatrix}
		b_{0} & b_{1} & b_{2} & \cdots & b_{n-1} \\
		0 & b_{0} & b_{1} & \cdots & b_{n-2} \\
		0 & 0 & b_{0} & \cdots & b_{n-3} \\
		\ddots & \ddots & \ddots & \vdots & \ddots \\
		0 & 0 & 0 & \cdots & b_{0}
	\end{pmatrix}  =\\
	& \begin{pmatrix}
		c_{0} & c_{1} & c_{2} & \cdots & c_{n-1} \\
		0 & c_{0} & c_{1} & \cdots & c_{n-2} \\
		0 & 0 & c_{0} & \cdots & c_{n-3} \\
		\ddots & \ddots & \ddots & \vdots & \ddots \\
		0 & 0 & 0 & \cdots & c_{0}
\end{pmatrix},
\end{align*}
where $$c_{i}=a_{0}\alpha^{0}(b_{i})+a_{1}\alpha^{1}(b_{i-1})+\cdots +a_{i}\alpha^{i}(b_{i}),~~ 1\leq i\leq n-1
.$$ We denote the elements of ${\rm T}_{n}(R, \alpha)$ by $(a_{0},a_{1},\ldots , a_{n-1})$. If $\alpha $ is the identity endomorphism, then ${\rm T}_{n}(R,\alpha )$ is a subring of upper triangular matrix ring ${\rm T}_{n}(R)$.

\begin{corollary}\label{2.20}
Let $R$ be a ring. Then, the following are equivalent:
\begin{enumerate}
\item
$R$ is a GSWNC ring.
\item
${\rm T}_{n}(R,\alpha )$ is a GSWNC ring.
\end{enumerate}
\end{corollary}

\begin{proof}
Choose
$$I:=\left\{
\left.
\begin{pmatrix}
		0 & a_{12} & \ldots & a_{1n} \\
		0 & 0 & \ldots & a_{2n} \\
		\vdots & \vdots & \ddots & \vdots \\
		0 & 0 & \ldots & 0
	\end{pmatrix} \right| a_{ij}\in R \quad (i\leq j )
	\right\}.$$
Then, one easily verifies that $I^{n}=0$ and $\dfrac{{\rm T}_{n}(R,\alpha )}{I} \cong R$. Consequently, Proposition \ref{2.13} applies to get the wanted result.
\end{proof}

Let $\alpha$ be an endomorphism of $R$. We denote by $R[x,\alpha ]$ the {\it skew polynomial ring} whose elements are the polynomials over $R$, the addition is defined as usual, and the multiplication is defined by the equality $xr=\alpha (r)x$ for any $r\in R$. Thus, there is a ring isomorphism $$\varphi : \dfrac{R[x,\alpha]}{\langle x^{n}\rangle }\rightarrow {\rm T}_{n}(R,\alpha),$$ given by $$\varphi (a_{0}+a_{1}x+\ldots +a_{n-1}x^{n-1}+\langle x^{n} \rangle )=(a_{0},a_{1},\ldots ,a_{n-1})$$ with $a_{i}\in R$, $0\leq i\leq n-1$. So, one finds that ${\rm T}_{n}(R,\alpha )\cong \dfrac{R[x,\alpha ]}{\langle  x^{n}\rangle}$, where $\langle x^{n}\rangle$ is the ideal generated by $x^{n}$. Also $R[[x, \alpha]]$ denotes the ring of skew formal power series over $R$; that is all formal power series in $x$ with coefficients from $R$ with multiplication defined by $xr = \alpha(r)x$ for all $r \in R$. On the other hand we know that, $\dfrac{R[x,\alpha ]}{\langle x^{n}\rangle}\cong \dfrac{R[[x,\alpha ]]}{\langle x^{n}\rangle}$.
We, thus, extract the following claim.

\begin{corollary}\label{2.21}
Let $R$ be a ring with an endomorphism $\alpha$ such that $\alpha (1)=1$. Then, the following are equivalent:
\begin{enumerate}
\item
$R$ is a GSWNC ring.
\item
$\dfrac{R[x,\alpha ]}{\langle x^{n}\rangle }$ is a GSWNC ring.
\item
$\dfrac{R[[x,\alpha ]]}{\langle x^{n}\rangle }$ is a GSWNC ring.
\end{enumerate}
\end{corollary}

\begin{corollary}
Let $R$ be a ring. Then, the following are equivalent:
\begin{enumerate}
\item
$R$ is a GSWNC ring.
\item
$\dfrac{R[x]}{\langle x^{n}\rangle }$ is a GSWNC ring.
\item
$\dfrac{R[[x]]}{\langle x^{n}\rangle }$ is a GSWNC ring.
\end{enumerate}
\end{corollary}

\begin{corollary}\label{2.57}
Let
$ R $
be a ring, and let
\begin{center}
$S_{n}(R)=\left\lbrace (a_{ij})\in T_{n}(R)\, | \, a_{11}=a_{22}=\cdots=a_{nn}\right\rbrace.$
\end{center}
Then, the following are equivalent:
\begin{enumerate}
\item	
$R$ is a GSWNC ring.
\item	
$S_{n}(R)$ is a GSWNC ring.
\end{enumerate}
\end{corollary}

\begin{proof}	
We assume $I=\{(a_{ij}) \in S_n(R) : a_{11}=0\}$, so evidently $I$ is a nil-ideal of $S_n(R)$ and $S_n(R)/I \cong R$.
\end{proof}

\begin{example}\label{2.24}
Let $R={\rm M}_2(\mathbb{Z}_2)$. It can be shown by simple calculation $$R = U(R) \cup {\rm Id}(R) \cup {\rm Nil}(R).$$ Thus, $R$ is a GSWNC ring, but is not strongly weakly nil-clean by \cite [Theorem 3.2]{1}.
\end{example}

\begin{example}
The ring ${\rm M}_2(\mathbb{Z}_{2^k})$ is a GSWNC ring for each $k \in \mathbb{N}$.
\end{example}

We now arrive at our first principal statement.

\begin{theorem}\label{2.25}
For any ring $R \neq 0$ and any integer $n \ge 3$, the ring ${\rm M}_n(R)$ is not GSWNC.
\end{theorem}

\begin{proof}
It suffices to show that ${\rm M}_3(R)$ is {\it not} a GSWNC ring by virtue of Lemma \ref{2.14}. To this target, consider the matrix
	$$A =\begin{pmatrix}
		1 & 1 & 0 \\
		1 & 0 & 0 \\
		0 & 0 & 0
\end{pmatrix} \notin U({\rm M}_3(R)).$$ Then we have, $$A \pm A^2 \notin {\rm Nil}({\rm M}_3(R)).$$ Therefore, Proposition \ref{2.11} guarantees that $R$ cannot be a GSWNC ring.
\end{proof}

Two consequences are these:

\begin{corollary}\label{2.35}
Let $R$ be a GSWNC ring. Then, for any $n>2$, there does not exist $0\neq e\in {\rm Id}(R)$ such that $eRe\cong {\rm M}_{n}(S)$ for some ring $S$.
\end{corollary}

\begin{proof}
Let us assume the opposite, namely that there exists $0\neq e\in {\rm Id}(R)$ such that $eRe\cong {\rm M}_{n}(S)$ for some ring $S$. Since $R$ is, by assumption, GSWNC, it follows from Lemma \ref{2.14} that the corner subring $eRe$ is GSWNC too, and hence ${\rm M}_{n}(S)$ is GSWNC as well. This, however, is a contradiction with Theorem \ref{2.25}.
\end{proof}

Recall that a set $\{e_{ij} : 1 \le i, j \le n\}$ of non-zero elements of $R$ is said to be a {\it system of $n^2$ matrix units} if $e_{ij}e_{st} = \delta_{js}e_{it}$, where $\delta_{jj} = 1$ and $\delta_{js} = 0$ for $j \neq s$. In this case, $e := \sum_{i=1}^{n} e_{ii}$ is an idempotent of $R$ and $eRe \cong {\rm M}_n(S)$, where $$S = \{r \in eRe : re_{ij} = e_{ij}r, \text{for all} i, j = 1, 2, . . . , n\}.$$
A ring $R$ is said to be Dedekind-finite if $ab=1$ give result that $ba=1$ for any $a$ and $b$ in $R$. In other words, all one-sided inverses in the ring are two-sided.

\begin{corollary}\label{2.36}
Every GSWNC ring is Dedekind-finite.
\end{corollary}

\begin{proof}
Suppose $R$ is a GSWNC ring. If we assume the reverse, namely that $R$ is {\it not} a Dedekind-finite ring, then there exist elements $a, b \in R$ such that $ab = 1$ but $ba \neq 1$. Putting $e_{ij} := a^i(1-ba)b^j$ and
$e :=\sum_{i=1}^{n}e_{ii}$, a routine verification shows that there will exist a non-zero ring $S$ such that $eRe \cong {\rm M}_n(S)$. However, according to Lemma \ref{2.14}, the corner $eRe$ is a GSWNC ring, so that ${\rm M}_n(S)$ must also be a GSWNC ring, thus contradicting Theorem \ref{2.25}.
\end{proof}

We continue by proving the following claims. 

\begin{lemma}\label{2.26}
Let ${\rm M}_2(R)$ be a GSWNC ring. Then, $R$ is a strongly weakly nil-clean ring.
\end{lemma}

\begin{proof}
Let $a \in R$. Then, $$A=\begin{pmatrix}
		a & 0 \\ 0 & 0
\end{pmatrix} \notin U({\rm M}_2(R)).$$ Thus by Proposition \ref{2.11}, we have $A\pm A^2 \in {\rm Nil}({\rm M}_2(R))$. So, $a \pm a^2 \in {\rm Nil}(R)$. Therefore, by \cite[Theorem 2.1]{1}, we can conclude that $R$ is a strongly weakly nil-clean ring.
\end{proof}

\begin{lemma}\label{2.27}
If $R$ is a local ring with nil $J(R)$, then $R$ is GSWNC.
\end{lemma}

\begin{proof}
Let $a\in R$ and $a\notin U(R)$. Since $R$ is local, $a\in J(R)$ and hence $a\in {\rm Nil}(R)$. So, $a$ is a nilpotent element, and thus it is strongly weaky nil-clean element.
\end{proof}

\begin{corollary}\label{2.50}
Let $R$ be a ring with only trivial idempotents. Then, $R$ is GSWNC if, and only if, $R$ is a local ring with $J(R)$ nil.
\end{corollary}

\begin{proof}
Let $R$ is a GSWNC ring, so $J(R)$ is nil by Lemma \ref{2.10}. Now, if $a \notin U(R)$, then we have either $a = q \pm 1$ or $a = q$, where $q \in Nil(R)$. Since $a$ is not an unit, it must be that $a = q$, implying $a = q \in Nil(R)$. Thus, by \cite[Proposition 19.3]{14}, $R$ is a local ring.
	
Now, conversely, suppose $R$ is a local ring with a nil Jacobson radical $J(R)$. So, for each $a \notin U(R)$, we have $a \in J(R) \subseteq Nil(R)$, whence $a$ is a nil-clean element.
\end{proof}

\begin{lemma}\label{2.55}
Suppose $R$ is a GSWNC ring and $2 \notin U(R)$. Then, either $2 \in {\rm Nil}(R)$ or $6 \in {\rm Nil}(R)$.
\end{lemma}

\begin{proof}
Assume that $2 \notin U(R)$. Then, there exists $e \in \text{Id}(R)$ and $q \in \text{Nil}(R)$ such that $2 = q \pm e$ and $qe=eq$. If $2 = e + q$, then $1 - e = q - 1 \in \text{Id}(R) \cap U(R)$. Thus, $e = 0$, which implies $2 = q \in \text{Nil}(R)$. Now, if $2 = -e + q$, we have $4 = (-e+q)(-e+q)$. Then, we have $4 = e + p$ for some $p \in \text{Nil}(R)$. Hence, $6 = 4 + 2 = p + q \in \text{Nil}(R)$ by noting that $pq = qp$.
\end{proof}

\begin{lemma}\label{2.29}
Let $R$ be a ring and $2 \in J(R)$. Then, the following two points are equivalent:
\begin{enumerate}
\item
$R$ is a GSWNC ring.
\item
$R$ is a GSNC ring.
\end{enumerate}
\end{lemma}

\begin{proof}
(ii) $\Longrightarrow$ (i). It follows at once.\\
(i) $\Longrightarrow$ (ii). Notice that $\dfrac{R}{J(R)}$ is of characteristic $2$, because $2 \in J(R)$, and so $a=-a$ for every $a \in \dfrac{R}{J(R)}$. That is why, $\dfrac{R}{J(R)}$ is a GSNC ring, and thus we can apply \cite[Proposition 2.7]{5} since $J(R)$ is nil in view of Proposition \ref{2.13}.
\end{proof}

\begin{lemma}\label{2.56}
Suppose $R$ is a ring such that $2 \notin U(R)$. Then, the following conditions are equivalent:
\begin{enumerate}
\item
$R$ is a GSWNC ring.
\item
Either $R$ is a GSNC ring or $R$ is strongly weakly nil-clean.
\end{enumerate}
\end{lemma}

\begin{proof}
$(ii) \Longrightarrow (i)$. This is obvious.\\
$(i) \Longrightarrow (ii)$. Thanks to Lemma \ref{2.55}, we have that either $2 \in \text{Nil}(R)$ or $6 \in \text{Nil}(R)$. If $2 \in \text{Nil}(R)$, it is clear that $R$ is a GSNC ring by Lemma \ref{2.29}. If $6 \in \text{Nil}(R)$ and, for $n \in \mathbb{N}$, we have $6^n = 0$, then $R \cong R_1 \oplus R_2$, where $R_1 = R / 2^nR$ and $R_2 = R / 3^nR$. However, Proposition \ref{2.5} tells us that $R_1$ and $R_2$ are strongly weakly nil-clean rings. Moreover, since $2 \in \text{Nil}(R_1)$, $R_1$ is a strongly nil-clean ring. Thus, \cite[Theorem 2.2]{1} implies that $R$ is a strongly weakly nil-clean ring, and hence $R$ is a GSWNC ring, as required.
\end{proof}

\begin{lemma}\label{2.30}
Let $R$ be a ring and. Then, the following are equivalent:
\begin{enumerate}
\item
$R$ is a strongly weakly nil-clean ring.
\item
$R$ is both WUU and GSWNC.
\end{enumerate}
\end{lemma}

\begin{proof}
It is evident by referring to \cite[Proposition 2.1]{30}.
\end{proof}

\begin{lemma}\label{2.31}
A ring $R$ is strongly nil-clean if, and only if,
\begin{enumerate}
\item
$R$ is GSWNC,
\item
$R$ is an UU-ring.
\end{enumerate}
\end{lemma}

\begin{proof}
It is apparent by combining \cite[Corollary 3.4]{6}, Corollary \ref{2.3} and Lemma \ref{2.10}.
\end{proof}

\begin{example}\label{2.33}
(i) For any ring $R$, the polynomial ring $R[x]$ is not GSWNC.

(ii) For any ring $R$, the power series ring $R[[x]]$ is not GSWNC.
\end{example}

\begin{proof}
(i) If we assume the contrary that $R[x]$ is GSWNC, then $R[x]$ is strongly weakly-clean, and  hence it is weakly clean. So, this is a contradiction.

(ii) Note the principal fact that the Jacobson radical of $R[[x]]$ is not nil. Thus, in view of Lemma \ref{2.10}, $R[[x]]$ need not be a GSWNC ring.
\end{proof}

Let $Nil_{*}(R)$ denote the prime radical of a ring $R$, i.e. the intersection of all prime ideals of $R$. We know that $Nil_{*}(R)$ is a nil-ideal of $R$.

A ring $R$ is called $2$-primal if $Nil_{*}(R)$ consists precisely of all the nilpotent elements of $R$. (For instance, reduced rings and commutative rings are both $2$-primal.)

\medskip

We now have at our disposal all the ingredients necessary to prove the following two major assertions.

\begin{theorem}\label{2.38}
Let R be a $2$-primal, local and strongly weakly nil-clean ring. Then, ${\rm M}_2(R)$ is a GSWNC ring.
\end{theorem}

\begin{proof}
Since $R$ is local and strongly weakly nil-clean, we have $R/J(R) \cong \mathbb{Z}_2$ or $R/J(R) \cong \mathbb{Z}_3$, so ${\rm M}_2(R/J(R))$ is a GSWNC ring. On the other hand, since $R$ is both $2$-primal and strongly weakly nil-clean, we may write $J(R) = {\rm Nil}(R) = {\rm Nil}_{\ast}(R)$, so we infer that
$${\rm M}_2(R/J(R)) = {\rm M}_2(R)/{\rm M}_2(J(R)) = {\rm M}_2(R)/{\rm M}_2({\rm Nil}_{\ast}(R)) = {\rm M}_2(R)/{\rm Nil}_{\ast}({\rm M}_2(R)).$$
Also, as ${\rm Nil}_{\ast}({\rm M}_2(R))$ is a nil-ideal, Proposition \ref{2.13} forces that ${\rm M}_2(R)$ is a GSWNC ring, as desired.
\end{proof}

\begin{theorem}\label{theorem 2.36}
Let $R$ be a ring. Then, the following conditions are equivalent for a semi-local ring:
\begin{enumerate}
\item
$R$ is a GSWNC ring.
\item
Either $R$ is a local ring with a nil Jacobson radical, or $R/J(R) \cong M_2(\mathbb{Z}_3)$ with a nil Jacobson radical, or $R/J(R) \cong M_2(\mathbb{Z}_2)$ with a nil Jacobson radical, or $R/J(R) \cong \mathbb{Z}_3 \times \mathbb{Z}_3$  with a nil Jacobson radical, or $R$  is a strongly weakly nil-clean ring.
\end{enumerate}
\end{theorem}

\begin{proof}
(ii) $\Longrightarrow$ (i). The proof is straightforward by combination of Lemma \ref{2.27}, Proposition \ref{2.13}. Also we know that $\mathbb{Z}_3 \times \mathbb{Z}_3$, $M_2(\mathbb{Z}_2)$ and $M_2(\mathbb{Z}_3)$ are GSWNC ring.\\
(i) $\Longrightarrow$ (ii). Since $R$ is semi-local, $R/J(R)$ is semi-simple, so we write $R/J(R) \cong \prod_{i=1}^{m}M_{n_i}(D_i)$, where each $D_i$ is a division ring. Moreover, the application of Proposition \ref{2.13} leads to $J(R)$ is nil, and $R/J(R)$ is a GSWNC ring. 

If $m = 1$, then Theorem \ref {2.26} applies to get that $n=1,2$. Now, if $n=1$, then $R/J(R) \cong D_1$. If, $n=2$, then $R/J(R) \cong M_2(D_1)$, so Theorem \ref {2.38} leads to either $D_1 \cong \mathbb{Z}_2$ or $D_1 \cong \mathbb{Z}_3$. Thus, $R/J(R) \cong M_2(\mathbb{Z}_2)$ or $R/J(R) \cong M_2(\mathbb{Z}_3)$. 

If $m = 2$, so $R/J(R) \cong M_{n_1}(D_1) \times M_{n_2}(D_2)$. As $R/J(R)$ is GSWNC, $M_{n_1}(D_1)$ and $M_{n_2}(D_2)$ are both strongly weakly nil-clean in view of Proposition \ref {2.5}. So, we arrive at $n_1=n_2=1$ in virtue of \cite[Corollary 3.1]{1} as well as $D_1 \cong \mathbb{Z}_2$ or $D_1 \cong \mathbb{Z}_3$, and $D_2 \cong \mathbb{Z}_2$ or $D_2 \cong \mathbb{Z}_3$. Consequently, one of the following isomorphisms holds: $R/J(R) \cong \mathbb{Z}_2 \times \mathbb{Z}_2,  \mathbb{Z}_2 \times \mathbb{Z}_3, \mathbb{Z}_3 \times \mathbb{Z}_3$. If  $R/J(R) \cong \mathbb{Z}_2 \times \mathbb{Z}_2,  \mathbb{Z}_2 \times \mathbb{Z}_3$, then $R/J(R)$ is strongly weakly nil-clean in accordance with \cite[Theorem 2.2]{1}. As $J(R)$ is nil, $R$ is strongly weakly nil-clean bearing in mind \cite[Lemma 3.1]{1}. 

If, however, $m > 2$, then Proposition \ref {2.6} employs to get that each $M_{n_i}(D_i)$ is strongly weakly nil-clean and at most one of them is not strongly nil-clean. Therefore, referring to \cite[Corollary 3.10]{2} and \cite[Proposition 2.1]{30}, for any $1 \le i \le m$, we detect $n_i=1$. Likewise, for any $1 \le i \le m$, we find that $D_i$ is strongly nil-clean and there exists an index $j$ such that $D_j$ is strongly weakly nil-clean. Finally, for any $1 \le i \le m$, we conclude that $D_i \cong \mathbb{Z}_2$ and that there exists $j$ with $D_j \cong \mathbb{Z}_2$ or $D_j \cong \mathbb{Z}_3$. So, in every case, $R/J(R)$ has to be strongly weakly nil-clean. But, as $J(R)$ is nil, $R$ is too strongly weakly nil-clean, and hence it is GSWNC, as wanted.
\end{proof}

Two more consequences are the following:

\begin{corollary}\label{2.39}
Let $R$ be a ring. Then, the following conditions are equivalent for a semi-simple ring:
\begin{enumerate}
\item
$R$ is a GSWNC ring.
\item
Either $R$ is a division ring, or $R\cong M_2(\mathbb{Z}_3)$, or $R\cong M_2(\mathbb{Z}_2)$, or $R\cong \mathbb{Z}_3 \times \mathbb{Z}_3$, or $R$  is a strongly weakly nil-clean ring.
\end{enumerate}
\end{corollary}

\begin{corollary}\label{2.40}
Let $R$ be a ring. Then, the following conditions are equivalent for an Artinian (in particular, a finite) ring:
\begin{enumerate}
\item
$R$ is a GSWNC ring.
\item
Either $R$ is a local ring with a nil Jacobson radical, or $R/J(R) \cong M_2(\mathbb{Z}_3)$ with a nil Jacobson radical, or $R/J(R) \cong M_2(\mathbb{Z}_2)$ with a nil Jacobson radical, or $R/J(R) \cong \mathbb{Z}_3 \times \mathbb{Z}_3$  with a nil Jacobson radical, or $R$  is a strongly weakly nil-clean ring.
\end{enumerate}
\end{corollary}

The next technicality is worthy of documentation.

\begin{lemma}\label{2.49}
Let $R$ be a GSWNC ring with $2 \in U(R)$ and, for every $u \in U(R)$, we have $u^2 = 1$. Then, $R$ is a commutative ring.
\end{lemma}

\begin{proof}
For any $u, v \in U(R)$, we have $u^2 = v^2 = (uv)^2 = 1$. Therefore, $uv = (uv)^{-1} = v^{-1}u^{-1} = vu$. Hence, the invertible elements commute with each other.
	
Now, we will show that $R$ is abelian. Indeed, for every idempotent $e$ in $R$ and $a \in R$, we know $2e-1\in U(R)$ and $(1+ea(1-e)) \in U(R)$. Since the invertible elements commute with each other, we have $2ea(1-e) = 2(1-e)ae = 0$. Since $2 \in U(R)$, we derive $ea(1-e) = (1-e)ae = 0$, which insures $ea = eae = ae$. Therefore, $R$ is abelian.
	
On the other side, since $1 + \text{Nil}(R) \subseteq U(R)$ and the invertible elements commute with each other, the nilpotent elements also commute with each other.
	
Now, let $x, y \in R$. We distinguish the following four basic cases:
\begin{enumerate}
\item $x, y \in U(R)$: Since the invertible elements commute with each other, it is clear that $xy = yx$.
\item $x, y \notin U(R)$: Since $R$ is a GSWNC ring, there exist $e, f \in \text{Id}(R)$ and $p, q \in \text{Nil}(R)$ such that $x = q \pm e$ and $y = p \pm f$. Then, we have four cases for $x$ and $y$. If $x=q+e$ and $y=p+f$ Thus,
\[xy = (q + e)(p + f) = qp + qf + ep + ef = pq + fq + pe + fe = (p + f)(q + e) = yx.\]
In the other three cases we will have similarly, $xy=yx$.
\item $x \in U(R)$ and $y \notin U(R)$: In this case, there exists $e \in \text{Id}(R)$ and $q \in \text{Nil}(R)$ such that $y = q \pm e$. Since $x(1 + q) = (1 + q)x$, we have $qx = xq$. Therefore,
\[xy = x(q \pm e) = xq \pm xe = qx \pm ex = (q \pm e)x = yx.\]
\item $x \notin U(R)$ and $y \in U(R)$: Similarly to case (iii), we can show that $xy = yx$.
\end{enumerate}
Then, $R$ is a commutative ring.
\end{proof}

Let $A$, $B$ be two rings and $M$, $N$ be $(A,B)$-bi-module and $(B,A)$-bi-module, respectively. Also, we consider the bilinear maps $\phi :M\otimes_{B}N\rightarrow A$ and $\psi:N\otimes_{A}M\rightarrow B$ that apply to the following properties.
$$Id_{M}\otimes_{B}\psi =\phi \otimes_{A}Id_{M},Id_{N}\otimes_{A}\phi =\psi \otimes_{B}Id_{N}.$$
For $m\in M$ and $n\in N$, define $mn:=\phi (m\otimes n)$ and $nm:=\psi (n\otimes m)$. Now the $4$-tuple $R=\begin{pmatrix}
	A & M\\
	N & B
\end{pmatrix}$ becomes to an associative ring with obvious matrix operations that is called a {\it Morita context ring}. Denote two-side ideals $Im \phi$ and $Im \psi$ to $MN$ and $NM$, respectively, that are called the {\it trace ideals} of the Morita context (compare with \cite{8} as well).

\medskip

We now proceed by showing a few assertions on various matrix extensions.

\begin{proposition}\label{2.41}
Let $R=\left(\begin{array}{ll}A & M \\ N & B\end{array}\right)$ be a Morita context ring such that $MN$ and $NM$ are nilpotent ideals of $A$ and $B$, respectively. If $R$ is a GSWNC ring, then $A$ and $B$ are strongly weakly nil-clean rings. The converse holds provided one of the $A$ or $B$ is strongly nil-clean and the other is strongly weakly nil-clean.
\end{proposition}

\begin{proof}
Since, $MN$ and $NM$ are nilpotent ideals of $A$ and $B$, respectively, $MN \subseteq J(A)$ and $NM\subseteq J(B)$. Therefore, in view of \cite{9}, we argue that $J(R)=\begin{pmatrix}
		J(A) & M \\
		N & J(B)
\end{pmatrix}$ and hence $\dfrac{R}{J(R)}\cong \dfrac{A}{J(A)}\times \dfrac{B}{J(B)}$.
Since $R$ is a GSWNC ring, then the factor $\dfrac{R}{J(R)}$ is a GSWNC ring and $J(R)$ is nil consulting with Proposition \ref{2.13}, so it follows that both $\dfrac{A}{J(A)}$ and $\dfrac{B}{J(B)}$ are strongly weakly nil-clean by Propposition \ref {2.5}. As $J(R)$ is nil, $J(A)$ and $J(B)$ are nil too. Thus, $A$ and $B$ are strongly weakly nil-clean by \cite[Lemma 3.1]{1}.\\
For the converse, let $A$ is a strongly nil-clean ring and $B$ is a strongly weakly nil-clean ring, we conclude that $A/J(A)$ is strongly nil-clean and $B/J(B)$ is strongly weakly nil-clean by a combination of \cite[Lemma 3.1]{1} and \cite[Corollary 3.17]{2}. So, $R/J(R)$ is strongly weakly nil-clean by \cite[Lemma 2.1]{1}. It then suffices to show that $J(R)$ is nil. To this target, suppose
	$r=\begin{pmatrix}
		a & m\\
		n & b
\end{pmatrix}\in J(R)$. Then, $a\in J(A)$, $b\in J(B)$. We know that both $J(A)$ and $J(B)$ are nil. Thus, we can find $n\in \mathbb{N}$ such that $a^{n}=0$ in $A$ and $b^{n}=0$ in $B$. So,
$$\begin{pmatrix}
		a & m\\
		n & b
\end{pmatrix}^{n+1}\subseteq \begin{pmatrix}
		MN & M\\
		N & NM
\end{pmatrix}.$$
Clearly,
	$$ \begin{pmatrix}
		MN & M\\
		N & NM
\end{pmatrix}^{2}=\begin{pmatrix}
		MN & (MN)M \\
		(NM)N & NM
\end{pmatrix}.$$
Moreover, for any $j\in \mathbb{N}$, one easily checks that
$$\begin{pmatrix}
		MN & M\\
		N & NM
\end{pmatrix}^{2j}=\begin{pmatrix}
		MN & (MN)M\\
		(NM)N  & NM
\end{pmatrix}^{j}=\begin{pmatrix}
		(MN)^{j} & (MN)^{j}M\\
		(NM)^{j}N & (NM)^{j}
\end{pmatrix}.$$
By hypothesis, we may assume that $(MN)^{p}=0$ in $A$ and $(NM)^{p}=0$ in $B$. Therefore,
$$\begin{pmatrix}
		MN & M\\
		N & NM
\end{pmatrix}^{2p}=0.$$
Consequently,
$\begin{pmatrix}
		a & m\\
		n & b
\end{pmatrix}^{2p(n+1)}=0$, and so $J(R)$ is indeed nil, as desired.	
\end{proof}

Now, let $R$, $S$ be two rings, and let $M$ be an $(R,S)$-bi-module such that the operation $(rm)s = r(ms$) is valid for all $r \in R$, $m \in M$ and $s \in S$. Given such a bi-module $M$, we can set

$$
{\rm T}(R, S, M) =
\begin{pmatrix}
	R& M \\
	0& S
\end{pmatrix}
=
\left\{
\begin{pmatrix}
	r& m \\
	0& s
\end{pmatrix}
: r \in R, m \in M, s \in S
\right\},
$$
where it forms a ring with the usual matrix operations. The so-stated formal matrix ${\rm T}(R, S, M)$ is called a {\it formal triangular matrix ring}. In Proposition \ref{2.41}, if we set $N =\{0\}$, then we will obtain the following corollary.

\begin{corollary}\label{2.42}
Let $R,S$ be rings and let $M$ be an $(R,S)$-bi-module. If ${\rm T}(R,S,M)$ is GSWNC, then $R$, $S$ are strongly weakly nil-clean.
The converse holds if one of the rings $R$ or $S$ is strongly nil-clean and the other is strongly weakly nil-clean.
\end{corollary}

Given a ring $R$ and a central elements $s$ of $R$, the $4$-tuple $\begin{pmatrix}
	R & R\\
	R & R
\end{pmatrix}$ becomes a ring with addition component-wise and with multiplication defined by
$$\begin{pmatrix}
	a_{1} & x_{1}\\
	y_{1} & b_{1}
\end{pmatrix}\begin{pmatrix}
	a_{2} & x_{2}\\
	y_{2} & b_{2}
\end{pmatrix}=\begin{pmatrix}
	a_{1}a_{2}+sx_{1}y_{2} & a_{1}x_{2}+x_{1}b_{2} \\
	y_{1}a_{2}+b_{1}y_{2} & sy_{1}x_{2}+b_{1}b_{2}
\end{pmatrix}.$$
This ring is denoted by ${\rm K}_{s}(R)$. A Morita context
$\begin{pmatrix}
	A & M\\
	N & B
\end{pmatrix}$ with $A=B=M=N=R$ is called a {\it generalized matrix ring} over $R$. It was observed by Krylov in \cite{10} that a ring $S$ is a generalized matrix ring over $R$ if, and only if, $S={\rm K}_{s}(R)$ for some $s\in {\rm Z}(R)$. Here $MN=NM=sR$, so $MN\subseteq J(A)\Longleftrightarrow s\in J(R)$, $NM\subseteq J(B)\Longleftrightarrow s\in J(R)$, and $MN$, $NM$  are nilpotent $\Longleftrightarrow s$ is a nilpotent.

\begin{corollary}\label{2.43}
Let $R$ be a ring and $s\in Z(R)\cap Nil(R)$. If $K_{s}(R)$ is a GSWNC ring, then $R$ is a strongly weakly nil-clean ring. The converse holds, provided $R$ is a strongly nil-clean ring.
\end{corollary}

Following Tang and Zhou (cf. \cite{11}), for any $n\geq 2$ and for $s\in Z(R)$, the $n\times n$ formal matrix ring over $R$ defined by $s$, and denoted by $M_{n}(R;s)$, is the set of all $n\times n$ matrices over $R$ with usual addition of matrices and with multiplication defined below:

\noindent For $(a_{ij})$ and $(b_{ij})$ in $M_{n}(R;s)$,
$$(a_{ij})(b_{ij})=(c_{ij}), \quad \text{where} ~~ (c_{ij})=\sum s^{\delta_{ikj}}a_{ik}b_{kj}.$$
Here, $\delta_{ijk}=1+\delta_{ik}-\delta_{ij}-\delta_{jk}$, where $\delta_{jk}$, $\delta_{ij}$, $\delta_{ik}$ are the Kroncker delta symbols.

\begin{corollary}\label{2.44}
Let $R$ be a ring and $s\in Z(R)\cap Nil(R)$. If $M_{n}(R;s)$ is a GSWNC ring, then $R$ is a strongly weakly nil-clean ring. The converse holds, provided $R$ is a strongly nil-clean ring.
\end{corollary}
\begin{proof}
If $n = 1$, then $M_n(R;s) = R$. So, in this case, there is nothing to prove. Let $n=2$. By the definition of $M_n(R;s)$, we have $M_2 (R;s) \cong K_{s^2} (R)$. Apparently, $s^2 \in Nil(R) \cap Z(R)$, so the claim holds for $n = 2$ with the help of Corollary \ref{2.43}.
	
To proceed by induction, assume now that $n>2$ and that the claim holds for $M_{n-1} (R;s)$. Set $A := M_{n-1} (R;s)$. Then, $M_n (R;s) =
\begin{pmatrix}
		A & M\\
		N & R
\end{pmatrix}$
is a Morita context, where $$M =
\begin{pmatrix}
		M_{1n}\\
		\vdots\\
		M_{n-1, n}
\end{pmatrix}
	\quad \text{and} \quad  N = (M_{n1} \dots M_{n, n-1})$$ with $M_{in} = M_{ni} = R$ for all $i = 1, \dots, n-1,$ and
\begin{align*}
		&\psi: N \otimes M \rightarrow N, \quad n \otimes m \mapsto snm\\
		&\phi : M \otimes N \rightarrow M, \quad  m \otimes n \mapsto smn.
\end{align*}
Besides, for $x =
\begin{pmatrix}
		x_{1n}\\
		\vdots\\
		x_{n-1, n}
\end{pmatrix}
	\in M$ and $y = (y_{n1} \dots y_{n, n-1}) \in N$, we write $$xy =
\begin{pmatrix}
		s^2x_{1n}y_{n1} & sx_{1n}y_{n2} & \dots & sx_{1n}y_{n, n-1}\\
		sx_{2n}y_{n1} & s^2x_{2n}y_{n2} & \dots & sx_{2n}y_{n, n-1}\\
		\vdots & \vdots &\ddots & \vdots\\
		sx_{n-1, n}y_{n1} & sx_{n-1, n}y_{n2} & \dots & s^2x_{n-1, n}y_{n, n-1}
\end{pmatrix} \in sA$$ and $$yx = s^2y_{n1}x_{1n} + s^2y_{n2}x_{2n} + \dots + s^2y_{n, n-1}x_{n-1, n} \in s^2 R.$$ Since $s$ is nilpotent, we see that $MN$ and $NM$ are nilpotent too. Thus, we obtain that $$\frac{M_n (R; s)}{J(M_n (R; s))} \cong \frac{A}{J (A)} \times \frac{R}{J (R)}.$$ Finally, the induction hypothesis and Proposition \ref{2.41} yield the claim after all.
\end{proof}

A Morita context $\begin{pmatrix}
	A & M\\
	N & B
\end{pmatrix}$ is called {\it trivial}, if the context products are trivial, i.e., $MN=0$ and $NM=0$. We now have
$$\begin{pmatrix}
	A & M\\
	N & B
\end{pmatrix}\cong T(A\times B, M\oplus N),$$
where
$\begin{pmatrix}
	A & M\\
	N & B
\end{pmatrix}$ is a trivial Morita context consulting with \cite{12}.

\begin{corollary}\label{2.45}
If the trivial Morita context
$\begin{pmatrix}
		A & M\\
		N & B
\end{pmatrix}$ is a GSWNC ring, then $A$, $B$ are strongly weakly nil-clean rings. The converse holds if one of the rings $A$ or $B$ is strongly nil-clean and the other is strongly weakly nil-clean.
\end{corollary}
\begin{proof}
It is apparent to see that the isomorphisms
$$\begin{pmatrix}
		A & M\\
		N & B
\end{pmatrix} \cong T(A\times B,M\oplus N) \cong \begin{pmatrix}
		A\times B & M\oplus N\\
		0 & A \times B
\end{pmatrix}$$ are fulfilled. Then, the rest of the proof follows combining Corollary \ref{2.17} and Proposition \ref{2.5}.
\end{proof}

\section{GSWNC Group Rings}

Let $R$ be a ring, $G$ a group, and, as usual, we denote by $RG$ the group ring of $G$ over $R$. Recall that a group $G$ is a {\it $p$-group} if every element of $G$ has order that is a power of the prime number $p$. Also, we recollect that a group is {\it locally finite} if each finitely generated subgroup is finite.

The homomorphism $\varepsilon :RG\rightarrow R$, defined by $\varepsilon (\displaystyle\sum_{g\in G}a_{g}g)=\displaystyle\sum_{g\in G}a_{g}$, is called the {\it augmentation map} of $RG$ and its kernel, denoted by $\Delta (RG)$, is called the {\it augmentation ideal} of $RG$.

\medskip

Our first three technicalities here are the following.

\begin{lemma} \label{3.1}
If $RG$ is a GSWNC ring, then $R$ is GSWNC.
\end{lemma}

\begin{proof}
We know that $RG/\Delta(RG) \cong R$. Therefore, by Proposition \ref{2.13}, it follows that $R$ must be a GSWNC ring, as stated.
\end{proof}

\begin{lemma}\label{3.2}
Let $R$ be a GSWNC ring with $p \in Nil(R)$ and let $G$ be a locally finite $p$-group, where $p$ is a prime. Then, the group ring $RG$ is GSWNC.
\end{lemma}

\begin{proof}
By \cite[Proposition 16]{25}, we know that $\Delta(RG)$ is a nil-ideal. Thus, since $RG/\Delta(RG) \cong R$, Proposition \ref{2.13} allows us to infer that $RG$ is a GSWNC ring.
\end{proof}

\begin{lemma}\label{3.5}
Let $f: R \to S$ be a non-zero epimorphism of rings with ${\rm Ker}(f) \cap {\rm Id}(R) = \{0\}$. Then, $R$ is a GSWNC ring if, and only if, $S$ is a GSWNC ring and ${\rm Ker}(f)$ is a nil-ideal of $S$.
\end{lemma}

\begin{proof}
Assume that $R$ is GSWNC. According to the relation $S \cong R/Ker(f)$ and Proposition \ref{2.13}, it is enough to prove that $Ker(f)$ is a nil-ideal. Now, let $a\in Ker(f)$. Thus, $a \notin U(R)$, so there exist $e\in Id(R)$ and $q\in Nil(R)$ with $a = q \pm e$ and $eq = qe$. Then, we have, $0 = f(a) = f(q) \pm f(e)$, and hence
$f(e) \in Id(S) \cap Nil(S) = \{0\}$. This shows that $e\in Id(R) \cap Ker(f) = \{0\}$, so $a = q \in Nil(R)$, therefore, $Ker(f)$ is a nil-ideal.\\
The converse is obvious by again Proposition \ref{2.13}.
\end{proof}

A valuable consequence is the following one.

\begin{corollary}\label{3.6}
Let $R$ be a ring and $G$ be a group, where $\Delta(RG)\cap Id(RG)=\{0\}$. Then, $RG$ is a GSWNC ring if, and only if, $R$ is a GSWNC ring and $\Delta(RG)$ is a nil-ideal of $RG$.
\end{corollary}

\begin{proof}
This is clear by an epimorphism $\varepsilon: RG \to R$ with $Ker(\varepsilon)= \Delta(RG)$ and Lemma \ref{3.5}.
\end{proof}

Our next three technicalities are as follows.

\begin{lemma}\label{3.7} \cite[Lemma $2$]{27}.
Let $p$ be a prime number with $p\in J(R)$. If $G$ is a locally finite $p$-group, then $\Delta(RG) \subseteq J(RG)$.
\end{lemma}

\begin{lemma}\label{3.8}
Let $R$ be a ring and let $G$ be a locally finite $p$-group, where $p$ is a prime number and $p\in J(R)$. Then, $RG$ is a GSWNC ring if, and only if, $R$ is a GSWNC ring and $\Delta(RG)$ is a nil-ideal of $RG$.
\end{lemma}

\begin{lemma}
Let $R$ be a GSWNC ring and let $G$ be a group such that $\Delta(RG) \subseteq J(RG)$. Then, $RG/J(RG)$ is a GSWNC ring.
\end{lemma}

\begin{proof}
We know that $RG = \Delta(RG) + R$, because $\Delta(RG) \subseteq J(RG)$, which implies that $RG = J(RG) + R$. Therefore,
$$R/(J(RG) \cap R) \cong (J(RG) + R)/J(RG) = RG/J(RG),$$
and since $R$ is GSWNC, the quotient $R/(J(RG) \cap R)$ is too GSWNC by Proposition \ref{2.13}. Then, we conclude that the factor $RG/J(RG)$ is a GSWNC ring, as required.
\end{proof}

An element $a \in R$ is called an {\it involution} if $a^2 = 1$. An element $w \in R$ is an unipotent if $1-w$ is a nilpotent. A ring $R$ is said to be an $IU$-ring if, for any $a$ in $R$, either $a$ or $-a$ is the sum of an involution and an unipotent.

\medskip

Our main statement here is the following.

\begin{theorem}\label{3.3}
Let $R$ be a ring such that $2 \notin U(R)$, and let $G$ be a group such that $RG$ is a GSWNC ring. Then, $G$ is a $2$-group, where $2 \in Nil(R)$.
\end{theorem}

\begin{proof}
Suppose $RG$ is a GSWNC ring. Invoking Lemma \ref{2.56}, we have that either $RG$ is a GSNC ring or $RG$ is a strongly weakly nil-clean ring. If $RG$ is a GSNC ring, one concludes from \cite[Theorem 3.7]{5} that $G$ is a $p$-group, where $p \in Nil(R)$. If $p$ is odd, this obviously contradicts $2 \in U(R)$, so it must be that $p = 2$.

If $RG$ is a strongly weakly nil-clean ring, then \cite[Theorem 1.15]{26} riches us that either $R$ is a strongly nil-clean ring and $G$ is a $2$-group, or $R$ is an $IU$ ring and $G$ is a $3$-group. If $R$ is a strongly nil-clean ring and $G$ is a $2$-group, there is nothing to prove, because from \cite[Proposition 3.14]{2} we have $2 \in Nil(R)$. However, if $R$ is an $IU$ ring, \cite [Lemma 4.2]{1} tells us that $2 \in U(R)$, which is a contradiction.
\end{proof}

The next construction is helpful to understand the complicated structure of GSWNC group rings.

\begin{example}\label{3.4}
Let $G$ be a non-trivial locally finite group and $2 \notin U(R)$. Then, $\mathbb{Z}_mG$ is a GSWNC ring if, and only if, $m=2^k$ for some positive integer $k$ and $G$ is a $2$-group.
\begin{proof}
Assume that $\mathbb{Z}_mG$ is a GSWNC ring, then $\mathbb{Z}_mG$ is GSNC or strongly weakly nil-clean by Lemma \ref{2.56}. If $\mathbb{Z}_mG$ is GSNC, then $G$ is a $p$-group, where $p \in Nil(R)$. If $p$ is odd, this contradicts $2 \in U(R)$, so $p = 2$. Also, we have $m=2^k$, because otherwise if for example $m=6$, $R=\mathbb{Z}_6$ and, $G=C_{2}$, then clearly the ring $RG$ is not GSNC (since if $RG$ is GSNC, then $\mathbb{Z}_6$ is GSNC by \cite[Proposition 2.7]{5} and this is contradiction). Now if $\mathbb{Z}_mG$ is strongly weakly nil-clean then by \cite[Example 1.17]{26}, we have three cases:\\
(1) $m=2^k$ for some positive integer $k$ and $G$ is a $2$-group.\\
(2) $m=3^k$ for some positive integer $k$ and $G$ is a $3$-group.\\
(3) $m=2^k3^l$ for some positive integers $k$, $l$ and $G$ is trivial.\\
According to the assumptions, cases (2) and (3) will not happen. So only case (1)  happens and the result is obtained.\\
Conversely, let $m=2^k$ for some positive integer $k$ and $G$ is a $2$-group. Then $\mathbb{Z}_mG$ is a strongly weakly nil-clean ring by \cite[Example 1.17]{26} and hence it is GSWNC.  	
\end{proof}
\end{example}

We, hereafter, denote the center of group $G$ by $Z(G)$.

\begin{lemma}\label{3.9}
Let $R$ be a ring with $2 \notin U(R)$ and $G$ be a group. If $RG$ is a GSNC ring, then $Z(G)$ is a $2$-group.	
\end{lemma}

\begin{proof}
Take $g \in Z(G)$. We know that every GSNC ring is GNC. So, by the \cite[Proposition 2.6]{29} and hypothesis, $2 \in Nil(RG)$. There is a ring epimorphism $RG \rightarrow (R/J(R))G$, so $(R/J(R))G$ is GSNC. Hence, without loss of generality, we can assume $J(R) = 0$. As an image of $RG$, $R$ is GSNC, so we have $2 = 0$ in $R$. Also, $1-g \in \Delta(RG)$, which implies that $1-g$ is not invertible. Therefore, there exist $e = e^2 \in RG$ and $q \in Nil(RG)$ such that $1-g = e + q$ and $eq=qe$. Since $Z(G)$ is abelian, we have $$1-e = g(1+g^{-1}q) \in Id(RG) \cap U(RG),$$ which guarantees that $e = 0$. Thus, $1-g = q \in Nil(RG)$. Hence, there is a natural number $k$ such that ${(1-g)}^{2^k}=0$. Since $char(RG)=2$, it must be that $g^{2^k}=1$. So, $Z(G)$ is a $2$-group, as claimed.
\end{proof}

The hyper-center of a group $G$, denoted by $H(G)$, is defined to be the union of the (transfinite) upper central series of the group $G$.	

\begin{lemma}\label{3.10}
Let $R$ be a ring with $2 \notin U(R)$ and $G$ be a group. If $RG$ is a GSNC ring, then $R$ is a GSNC ring and the hyper-center $H(G)$ is a $2$-group.
\end{lemma}

\begin{proof}
Utilizing a completely similar method as that of Theorem 2.6 in \cite{31}, the claim follows.
\end{proof}

Recall that a group $G$ is called {\it nilpotent} if it has a normal series $$1=G_{0}\leq G_{1}\leq \dots \leq G_{n}=G$$ such that $$G_{i+1}/G_{i}\leq Z(G/G_{i}),$$ for $i=0,1, \dots, n-1$. Moreover, such a series is called a {\it central series} for $G$.

\begin{lemma}\label{3.11}
Let $R$ be a ring with $2 \notin U(R)$ and $G$ be a nilpotent group. Then, $RG$ is a GSNC ring, if, and only if, $R$ is a GSNC ring and $G$ is a $2$-group.
\end{lemma}

\begin{proof}
$(\Rightarrow)$. It follows immediately from Lemma \ref{3.10}.
	
$(\Leftarrow)$. The hypothesis and \cite[Proposition 2.6]{29} yields that $2 \in Nil(R)$. In accordance with \cite[Proposition 16]{25}, we know that $\Delta(RG)$ is a nil-ideal.
Now the isomorphism $RG/\Delta(RG) \cong R$ and \cite[Proposition 4.7(i)]{5} allows us to get that $RG$
is a GSNC ring.
\end{proof}

We are now prepared to proceed by establishing the following affirmation.

\begin{theorem}\label{3.12}
Let $R$ be a ring such that $2 \notin U(R)$, and $G$ be a nilpotent group. Then, $RG$ is a GSWNC ring if, and only if, $R$ is a GSNC ring and $G$ is a $2$-group.
\end{theorem}

\begin{proof}
($\Rightarrow$). Suppose $RG$ is a GSWNC ring. From Lemma \ref{2.56}, we deduce that either $RG$ is a GSNC ring or $RG$ is a strongly weakly nil-clean ring. If, foremost, $RG$ is a GSNC ring, one concludes from Lemma \ref{3.11} that $R$ is a GSNC ring and $G$ is a $2$-group. If $RG$ is a strongly weakly nil-clean ring, then \cite[Theorem 1.15]{26} riches us that either $R$ is a strongly nil-clean ring and $G$ is a $2$-group, or $R$ is an $IU$-ring and $G$ is a $3$-group. If $R$ is a strongly nil-clean ring and $G$ is a $2$-group, there is nothing left to prove, because from \cite[Proposition 3.14]{2}, we have $2 \in Nil(R)$. However, if $R$ is an $IU$-ring and $G$ is a $3$-group, we derive with the aid of \cite[Lemma 4.2]{1} that $2 \in U(R)$, which is a contradiction.
	
($\Leftarrow$). It follows directly from Lemma \ref{3.11}.
\end{proof}

\medskip
\medskip


\medskip

\noindent{\bf Funding:} The work of the first-named author, P.V. Danchev, is partially supported by the Junta de Andaluc\'ia under Grant FQM 264. All other four authors are supported by Bonyad-Meli-Nokhbegan and receive funds from this foundation.

\end{document}